\documentclass[12pt]{article}

\usepackage{amsmath,amssymb,amsbsy,amsfonts,amsthm,latexsym,
                        amsopn,amstext,amsxtra,euscript,amscd,color}
\usepackage{color,xcolor}
\usepackage{latexsym}
\usepackage{cite}
\usepackage{amsfonts}
\usepackage{amsfonts,mathrsfs}
\usepackage{graphicx}
\usepackage{psfrag}
\usepackage{subfigure}
\usepackage{url}
\usepackage{stfloats}
\usepackage{amsmath}

\usepackage{hyperref}
\hypersetup{colorlinks=true}

\newtheorem{theorem}{Theorem}
\newtheorem{lemma}{Lemma}
\newtheorem{corollary}{Corollary}

\newcommand{\quash}[1]{}

\begin{document}

\title{Correlation measures of  binary sequences derived from
Euler quotients}

\author{Huaning Liu$^1$, Zhixiong Chen$^2$, Chenhuang Wu$^2$\\
1.  Research Center for Number Theory and Its Applications,\\
   School of Mathematics, Northwest University,\\
       Xi'an 710127, Shaanxi, P. R. China \\
2. Provincial Key Laboratory of Applied Mathematics,\\
 Putian University, Putian 351100, Fujian, P. R. China                            
}

\maketitle

\begin{abstract}
Fermat-Euler quotients arose
from the study of the first case of Fermat's Last Theorem, and
have numerous applications in number theory. Recently they were studied from the cryptographic aspects by constructing many pseudorandom binary sequences, whose
linear complexities and trace representations were calculated. In this work,
we further study their correlation measures
by using the approach based on Dirichlet characters, Ramanujan sums and Gauss sums.
Our results show that the $4$-order correlation measures of
these sequences are very large. Therefore they may not be suggested for cryptography.
\end{abstract}


\section{Introduction}

Let $\mathcal{S}=(s_0,s_1,...)$ be a binary sequence
over $\mathbb{F}_2=\{0,1\}$ and $k$ a positive integer.
Mauduit and S\'ark\"ozy \cite{MauduitS1997} introduced the \emph{($N$-th) correlation
  measure of order $k$} for the first $N$ terms of $\mathcal{S}$, which is defined as
$$
C_{k}(\mathcal{S},N)=\max_{U,D}\left|\sum^{U-1}_{n=0}(-1)^{s_{n+d_{1}}+s_{n+d_{2}}+\ldots+s_{n+d_{k}}}\right|,
$$
where the maximum is taken over all $U \leq N - k + 1$ and
$D = (d_1, d_2, \ldots , d_k)$ with
integers $0 \leq d_1 < d_2 < \ldots < d_k \le N - U$.

From the viewpoint of cryptography, it is excepted that the  $N$-th correlation
  measure of order $k$ of sequences is as ``small" (in
terms of $N$, in particular, is $o(N)$ as $N\rightarrow
\infty$) as possible.
It was shown in \cite{CassaigneMS2002} that for a ``truely" random
sequence  $\mathcal{S}$, $C_{k}(\mathcal{S},N)$ (for some
fixed $k$) is around $N^{1/2}$ with ``near 1" probability.

Additionally,
if $\mathcal{S}$ is $T$-periodic (in this case, we will denote $\mathcal{S}$ as $\mathcal{S}=(s_0,s_1,\ldots,s_{T-1})$),
we use the following definition for
the \emph{periodic correlation measure of order $k$} of $\mathcal{S}$,
$$
\theta_k(\mathcal{S}) = \max_{D}\left|\sum^{T-1}_{n=0}(-1)^{s_{n+d_{1}}+s_{n+d_{2}}+\ldots+s_{n+d_{k}}}\right|,
$$
where $D=(d_1,\ldots, d_k)$, with
$0\le d_1 < d_2 < \ldots < d_k < T$. It is clear  that $\theta_2(\mathcal{S})$ is the (classic) auto-correlation of $\mathcal{S}$.

In this work, we mainly consider the periodic correlation measure of order $k$ for some binary
sequences derived from Euler quotients studied recently.\\

 Let $p$ be a prime and let $n$ be an integer with $\gcd(n,p)=1$. From the Fermat's Little
Theorem we know that $n^{p-1}\equiv 1\ (\bmod\ p)$. Then the  Fermat quotient $Q_p(n)$ is defined as
$$
Q_{p}(n)=\frac{n^{p-1}-1}{p} \ (\bmod\ p),\quad 0\leq Q_{p}(n)< p.
$$
We also define $Q_{p}(n)=0$ if $\gcd(n,p)>1$. Fermat quotients arose
from the study of the first case of Fermat's Last Theorem, and
have many applications in number
theory (see
\cite{AlyW2011,Chang2012,ChenW2014,GomezW2012,OstafeS2011,Shparlinski2011_1,Shparlinski2011_2,Shparlinski2011_3}
for details). Define the $p^2$-periodic binary sequence $\displaystyle
\mathbf{\overline{s}}=\left(\overline{s}_0, \overline{s}_1,\cdots, \overline{s}_{p^2-1}\right)$ by
\begin{eqnarray*}
\overline{s}_t=\left\{\begin{array}{ll}\displaystyle 0, & \hbox{if \ } 0\leq
\frac{Q_p(t)}{p}<\frac{1}{2}, \\
\displaystyle 1,  & \hbox{if \ } \frac{1}{2}\leq \frac{Q_p(t)}{p}<1.
\end{array}
\right.
\end{eqnarray*}
The second author (partially with other co-authors) studied the well-distribution measure
and correlation measure of order $2$ of $\mathbf{\overline{s}}$
by using estimates for exponential sums of Fermat quotients in \cite{ChenOW2010},
the linear complexity of $\mathbf{\overline{s}}$ in \cite{ChenD2013,ChenHD2012}, and the trace representation of $\mathbf{\overline{s}}$ by determining the
defining pairs of all binary characteristic sequences of cosets in \cite{Chen2014}.    In \cite{LiuL2022}
the first author with other co-author showed that the $4$-order correlation measure of $\mathbf{\overline{s}}$ is very large.

Let $m\geq 2$ be an odd number and let $n$ be an integer coprime to $m$. The Euler's theorem says that
$n^{\phi(m)}\equiv 1\ (\bmod\ m)$, where $\phi$ is the Euler's totient function. Then the Euler
quotient $Q_{m}(n)$ is defined as
$$
Q_{m}(n)=\frac{n^{\phi(m)}-1}{m} \ (\bmod\ m),\quad 0\leq Q_{m}(n)< m.
$$
We also define $Q_{m}(n)=0$ if $\gcd(n,m)>1$. Agoh, Dilcher and Skula \cite{AgohDS1997}
studied the detailed properties of Euler quotients. For example, from Proposition 2.1 of
\cite{AgohDS1997} we have
\begin{eqnarray}\label{Eulerquotients1}
Q_{m}(n_1n_2)\equiv Q_{m}(n_1)+Q_{m}(n_2)\ (\bmod\ m) \quad \hbox{for} \quad \gcd(n_1n_2,m)=1,
\end{eqnarray}
and
\begin{eqnarray}\label{Eulerquotients2}
Q_{m}(n+cm)\equiv Q_{m}(n)+cn^{-1}\phi(m)\ (\bmod\ m) \quad \hbox{for} \quad \gcd(n,m)=1.
\end{eqnarray}

Recently many binary sequences
were constructed from Euler quotients. For example, let $m=p^{\tau}$ for a fixed number $\tau\geq 1$, the $p^{\tau+1}$-periodic sequence
$\displaystyle\mathbf{\widetilde{s}}=\left(\widetilde{s}_0, \widetilde{s}_1,\cdots, \widetilde{s}_{p^{\tau+1}-1}\right)$ is defined by
\begin{eqnarray}\label{Sequence-Eulerquotients-1}
\widetilde{s}_t=\left\{\begin{array}{ll}\displaystyle 0, & \hbox{if \ } 0\leq
\frac{Q_{p^{\tau}}(t)}{p^{\tau}}<\frac{1}{2}, \\
\displaystyle 1,  & \hbox{if \ } \frac{1}{2}\leq \frac{Q_{p^{\tau}}(t)}{p^{\tau}}<1.
\end{array}
\right.
\end{eqnarray}
The linear complexity of $\mathbf{\widetilde{s}}$ had been investigated in
\cite{DuCH2012} and the trace representation of $\mathbf{\widetilde{s}}$ was given
in \cite{ChenDM2015}.

Moreover, let $m=pq$ be a product of two distinct odd primes $p$ and $q$ with $p\mid (q-1)$, the $pq^2$-periodic sequence
$\displaystyle\mathbf{\widehat{s}}=\left(\widehat{s}_0, \widehat{s}_1,\cdots, \widehat{s}_{pq^2-1}\right)$ is defined by
\begin{eqnarray}\label{Sequence-Eulerquotients-2}
\widehat{s}_t=\left\{\begin{array}{ll}\displaystyle 0, & \hbox{if \ } 0\leq
\frac{Q_{pq}(t)}{pq}<\frac{1}{2}, \\
\displaystyle 1,  & \hbox{if \ } \frac{1}{2}\leq \frac{Q_{pq}(t)}{pq}<1.
\end{array}
\right.
\end{eqnarray}
Very recently the minimal polynomials and linear complexities were determined in \cite{ZhangGZ2020}
for $\mathbf{\widehat{s}}$, and the
trace representation of $\mathbf{\widehat{s}}$ has been given in \cite{ZhangHFZ2021}
provided that $2^{q-1}\not\equiv 1 \ (\bmod\ q^2)$.\\

In this work, we shall further study the (periodic) correlation measures of
$\mathbf{\widetilde{s}}$ and $\mathbf{\widehat{s}}$ by using the approach based on
Dirichlet characters, Ramanujan sums and Gauss sums. We state below the main result.

\begin{theorem}\label{Theorem-Sequence-Eulerquotients-generalized}
Let $m\geq 2$ be an odd number. Suppose that $Q_{m}(n)$ is $km$-periodic with $k>3$ and $k\mid m$.
Define the $km$-periodic sequence
$\displaystyle\mathbf{s}=\left(s_0, s_1,\cdots, s_{km-1}\right)\in\{0,1\}^{km}$ by
\begin{eqnarray}\label{Sequence-Eulerquotients-0}
s_t=\left\{\begin{array}{ll}\displaystyle 0, & \hbox{if \ } 0\leq
\frac{Q_{m}(t)}{m}<\frac{1}{2}, \\
\displaystyle 1,  & \hbox{if \ } \frac{1}{2}\leq \frac{Q_{m}(t)}{m}<1.
\end{array}
\right.
\end{eqnarray}
Then we have
$$
\sum_{t=0}^{km-1}(-1)^{s_t+s_{t+m}+s_{t+2m}+s_{t+3m}}
=km-\frac{2}{3}k\phi(m)+O\left(\phi(m)\phi(k)^{-1}k^{\frac{3}{2}}(\log m)^4\right).
$$
\end{theorem}

Taking special values of $m$ and $k$ in Theorem \ref{Theorem-Sequence-Eulerquotients-generalized},
we immediately get the correlation measures of
$\mathbf{\widetilde{s}}$ and $\mathbf{\widehat{s}}$.

\begin{corollary}
Let $p>2$ be a prime and let $\tau\geq 1$ be a fixed integer. Let the $p^{\tau+1}$-periodic sequence
$\displaystyle\mathbf{\widetilde{s}}=\left(\widetilde{s}_0, \widetilde{s}_1,\cdots, \widetilde{s}_{p^{\tau+1}-1}\right)$ be defined as in
$(\ref{Sequence-Eulerquotients-1})$.
Then we have
$$
\sum_{t=0}^{p^{\tau+1}-1}(-1)^{\widetilde{s}_t+\widetilde{s}_{t+p^{\tau}}
+\widetilde{s}_{t+2p^{\tau}}+\widetilde{s}_{t+3p^{\tau}}}
=\frac{1}{3}p^{\tau+1}+O\left(p^{\tau+\frac{1}{2}}(\log p^{\tau})^4\right).
$$
\end{corollary}

\begin{corollary}
Let $p$ and $q$ be two distinct odd primes with $p\mid (q-1)$, and let
the $pq^2$-periodic sequence $\displaystyle\mathbf{\widehat{s}}=\left(\widehat{s}_0, \widehat{s}_1,\cdots, \widehat{s}_{pq^2-1}\right)$ be defined as in $(\ref{Sequence-Eulerquotients-2})$.
Then we have
$$
\sum_{t=0}^{pq^2-1}(-1)^{\widehat{s}_t+\widehat{s}_{t+pq}+\widehat{s}_{t+2pq}+\widehat{s}_{t+3pq}}
=\frac{1}{3}pq^2+\frac{2}{3}q^2+O\left(pq^{\frac{3}{2}}(\log pq)^4\right).
$$
\end{corollary}

Our results indicate that the correlation measures of order $4$ of $\displaystyle\mathbf{\widetilde{s}}$
and $\displaystyle\mathbf{\widehat{s}}$ are very large.
Therefore these sequences may not be suggested for cryptography.\\

To prove Theorem \ref{Theorem-Sequence-Eulerquotients-generalized}, we introduce basic
properties of Dirichlet characters, Ramanujan sums and Gauss sums,
and then prove two lemmas on the mean values of characters sums in Sect. 2.
We express $(-1)^{s_t}$ in terms of character sums in Sect. 3 to finish the proof of Theorem \ref{Theorem-Sequence-Eulerquotients-generalized} by using the results
showed in Sect. 2.

\bigskip

\section{Dirichlet characters and Gauss sums}

\quad \ Let $N>1$ be an integer. The Ramanujan sum is denoted by
$$
c_{N}(n)=\mathop{\sum_{t=0}^{N-1}}_{\gcd(t, N)=1}e_N(tn),
$$
where $e_{N}(x)=\hbox{e}^{2\pi\sqrt{-1}x/N}$. We have
\begin{eqnarray}\label{Ramanujan-sum}
c_{N}(n)=\mu\left(\frac{N}{\gcd(n,N)}\right)\phi(N)\phi^{-1}\left(\frac{N}{\gcd(n,N)}\right),
\end{eqnarray}
where $\mu$ is the M\"{o}bius function.

If the function $\chi$ satisfies the following conditions:

\qquad (i). \ $\chi(n_1n_2)=\chi(n_1)\chi(n_2)$ for all integers $n_1$, $n_2$,

\qquad (ii). \ $\chi(n+N)=\chi(n)$ for all integers $n$,

\qquad (iii). \ $\chi(n)=0$ for $\gcd(n,N)>1$, \\
then $\chi$ is one of the Dirichlet characters modulo $N$. When $\chi(n)=1$
for all $n$ with $\gcd(n,N)=1$ we say $\chi$ is the trivial character modulo $N$.
The integer $d$ is called an induced modulus for $\chi$ if
$\chi(a)=1$ whenever $\gcd(a,N)=1$ and $a\equiv 1\ (\bmod\ d)$.
A Dirichlet character $\chi$ mod $N$ is said to be primitive mod $N$ if it has no
induced modulus $d<N$. The smallest induced modulus $d$ for $\chi$ is called the conductor of $\chi$.
Every non-trivial character $\chi$ modulo $N$ can be uniquely
written as $\chi=\chi_0\chi^{*}$, where $\chi_0$ is the trivial character modulo $N$
and $\chi^{*}$ is the primitive character modulo the conductor of $\chi$.

For a Dirichlet character $\chi$ mod $N$, the Gauss sum associated with $\chi$ is defined by
$$
G(n,\chi)=\sum_{t=0}^{N-1}\chi(t)e_N(tn).
$$
Let $N^{*}$ be the conductor for $\chi$ and let $\chi^{*}$ be the induced primitive character.
Let $N_1$ be the maximal divisor of $N$ such that $N_1$ and $N^{*}$ have the same prime divisors.
Then we have
\begin{eqnarray}\label{Gauss-sum}
G(n,\chi)=\left\{\begin{array}{l}
\left(\chi^*\right)^{-1}\left(\frac{n}{\gcd(n,N)}\right)\chi^*\left(\frac{N}{N^*\gcd(n,N)}\right)
\mu\left(\frac{N}{N^*\gcd(n,N)}\right)  \\
\quad\times\phi(N)\phi^{-1}\left(\frac{N}{\gcd(n,N)}\right)G(1,\chi^*),   \qquad\hbox{if \ } N^*=\frac{N_1}{\gcd(n, N_1)}, \\
0, \qquad\qquad\qquad\qquad\qquad\qquad\qquad\qquad \hbox{if \ } N^*\neq\frac{N_1}{\gcd(n, N_1)}.
\end{array}
\right.
\end{eqnarray}
See Chapter 8 of \cite{Apostol1976} or Chapter 1 of \cite{PanP1992} for
more details of Dirichlet characters, Ramanujan sums and Gauss sums.

Now we prove two lemmas on the mean values of characters sums.

\begin{lemma}\label{Lemma-charactersums}
Let $m\geq 2$ be an odd number and let $k>3$ be an integer with $k\mid m$.
Let $\chi$ be a Dirichlet
character modulo $km$ with $\chi^m$ being trivial. For integers
$a_1$, $a_2$, $a_3$ and $a_4$ we have
\begin{eqnarray*}
&&\sum_{t=0}^{km-1}\chi\left(t^{a_1}(t+m)^{a_2}(t+2m)^{a_3}(t+3m)^{a_4}\right) \\
&&=\left\{\begin{array}{ll}
\displaystyle k\phi(m), & \hbox{if \ } m\mid (a_1+a_2+a_3+a_4) \hbox{ \ and \ }
k\mid (a_2+2a_3+3a_4), \\
\displaystyle O\left(\phi(m)\phi(k)^{-1}k^{\frac{3}{2}}\right), & \hbox{otherwise}.
\end{array}
\right.
\end{eqnarray*}
\end{lemma}

\begin{proof} By the properties of residue systems we get
\begin{eqnarray*}
&&\sum_{t=0}^{km-1}\chi\left(t^{a_1}(t+m)^{a_2}(t+2m)^{a_3}(t+3m)^{a_4}\right) \nonumber\\
&&=\mathop{\sum_{y=0}^{m-1}}_{\gcd(y,m)=1}\sum_{z=0}^{k-1}
\chi\left((y+zm)^{a_1}(y+zm+m)^{a_2}(y+zm+2m)^{a_3}(y+zm+3m)^{a_4}\right) \nonumber\\
&&=\mathop{\sum_{y=0}^{m-1}}_{\gcd(y,m)=1}\sum_{z=0}^{k-1}
\chi\left((y^{a_1}+a_1y^{a_1-1}zm)(y^{a_2}+a_2y^{a_2-1}(z+1)m)\right) \nonumber\\
&&\qquad\times\chi\left((y^{a_3}+a_3y^{a_3-1}(z+2)m)
(y^{a_4}+a_4y^{a_4-1}(z+3)m)\right)\nonumber\\
&&=\mathop{\sum_{y=0}^{m-1}}_{\gcd(y,m)=1}\chi\left(y^{a_1+a_2+a_3+a_4}\right)\nonumber\\
&&\qquad\times\sum_{z=0}^{k-1}\chi\left((1+a_1y^{-1}zm)(1+a_2y^{-1}(z+1)m)
(1+a_3y^{-1}(z+2)m)(1+a_4y^{-1}(z+3)m)\right).
\end{eqnarray*}
Clearly $\chi\left(1+nm\right)$ is a primitive additive character modulo $k$,
so there is uniquely an integer $\beta$ such that $1\leq \beta\leq k$, $\gcd(\beta,k)=1$ and
$\chi\left(1+nm\right)=e_k(\beta n)$. Hence,
\begin{eqnarray*}
&&\sum_{t=0}^{km-1}\chi\left(t^{a_1}(t+m)^{a_2}(t+2m)^{a_3}(t+3m)^{a_4}\right) \nonumber\\
&&=\mathop{\sum_{y=0}^{m-1}}_{\gcd(y,m)=1}\chi\left(y^{a_1+a_2+a_3+a_4}\right) \nonumber\\
&&\qquad\times\sum_{z=0}^{k-1}e_k\left(\beta(a_1y^{-1}z+a_2y^{-1}(z+1)
+a_3y^{-1}(z+2)+a_4y^{-1}(z+3))\right) \nonumber\\
&&=\left\{\begin{array}{ll}
\displaystyle k\mathop{\sum_{y=0}^{m-1}}_{\gcd(y,m)=1}\chi^{-(a_1+a_2+a_3+a_4)}\left(y\right)
e_k\left(\beta(a_2+2a_3+3a_4)y\right), & \hbox{if \ } k\mid (a_1+a_2+a_3+a_4),\\
\displaystyle 0, & \hbox{if \ } k\nmid (a_1+a_2+a_3+a_4).
\end{array}
\right.
\end{eqnarray*}

We know that $\chi^{a_1+a_2+a_3+a_4}$ is a multiplicative character modulo $m$
if $k\mid a_1+a_2+a_3+a_4$. Then from (\ref{Ramanujan-sum}), (\ref{Gauss-sum})
and the properties of Ramanujan sums and Gauss sums we get
\begin{eqnarray*}
&&\sum_{t=0}^{km-1}\chi\left(t^{a_1}(t+m)^{a_2}(t+2m)^{a_3}(t+3m)^{a_4}\right) \\
&&=\left\{\begin{array}{ll}
\displaystyle k\phi(m), & \hbox{if \ } m\mid (a_1+a_2+a_3+a_4) \hbox{ \ and \ }
k\mid (a_2+2a_3+3a_4), \\
\displaystyle O\left(\phi(m)\phi(k)^{-1}k^{\frac{3}{2}}\right), & \hbox{otherwise}.
\end{array}
\right.
\end{eqnarray*}
\end{proof}

\begin{lemma}\label{Lemma-exponentialsums}
Let $m\geq 2$ be an odd number and let $k>3$ be an integer with $k\mid m$.  Define
\begin{eqnarray*}
\Xi_{m,k}&:=&\mathop{\mathop{
\sum_{1\leq|a_1|, |a_2|, |a_3|, |a_4|\leq\frac{m-1}{2}}
}_{a_1+a_2+a_3+a_4\equiv 0\ (\bmod m)}}_{a_2+2a_3+3a_4\equiv 0\ (\bmod k)}
~~\sum_{l_1=\frac{m+1}{2}}^{m-1}e_m\left(-a_1l_1\right)
\sum_{l_2=\frac{m+1}{2}}^{m-1}e_m\left(-a_2l_2\right)\nonumber\\
&&\times\sum_{l_3=\frac{m+1}{2}}^{m-1}e_m\left(-a_3l_3\right)
\sum_{l_4=\frac{m+1}{2}}^{m-1}e_m\left(-a_4l_4\right).
\end{eqnarray*}
Then we have
\begin{eqnarray*}
\Xi_{m,k}=\frac{1}{48}m^4+O\left(\frac{m^4(\log m)^3}{k}\right).
\end{eqnarray*}
\end{lemma}

\begin{proof}
For absolute constant $c>0$ we get
\begin{eqnarray*}
&&\sum_{ck\leq|a_1|\leq\frac{m-1}{2}}\mathop{
\sum_{1\leq|a_2|, |a_3|, |a_4|\leq\frac{m-1}{2}}
}_{a_1+a_2+a_3+a_4\equiv 0\ (\bmod m)}
\left|\sum_{l_1=\frac{m+1}{2}}^{m-1}e_m\left(-a_1l_1\right)\right|
\cdot\left|\sum_{l_2=\frac{m+1}{2}}^{m-1}e_m\left(-a_2l_2\right)\right|\nonumber\\
&&\qquad\times\left|\sum_{l_3=\frac{m+1}{2}}^{m-1}e_m\left(-a_3l_3\right)\right|\cdot
\left|\sum_{l_4=\frac{m+1}{2}}^{m-1}e_m\left(-a_4l_4\right)\right| \nonumber\\
&&\ll\sum_{1\leq|a_2|\leq\frac{m-1}{2}}\frac{m}{|a_2|}
\sum_{1\leq|a_3|\leq\frac{m-1}{2}}\frac{m}{|a_3|}\sum_{1\leq|a_4|\leq\frac{m-1}{2}}
\frac{m}{|a_4|}\mathop{\sum_{ck\leq|a_1|\leq\frac{m-1}{2}}}_{a_1+a_2+a_3+a_4\equiv 0\ (\bmod m)}
\frac{m}{k} \nonumber\\
&&\ll\frac{m^4(\log m)^3}{k}.
\end{eqnarray*}
Hence,
\begin{eqnarray*}
\Xi_{m,k}&:=&\sum_{1\leq|a_1|, |a_2|\leq\frac{5k}{32}}\mathop{\mathop{
\sum_{1\leq|a_3|, |a_4|\leq\frac{k}{32}}
}_{a_1+a_2+a_3+a_4\equiv 0\ (\bmod m)}}_{a_2+2a_3+3a_4\equiv 0\ (\bmod k)}
\sum_{l_1=\frac{m+1}{2}}^{m-1}e_m\left(-a_1l_1\right)
\sum_{l_2=\frac{m+1}{2}}^{m-1}e_m\left(-a_2l_2\right)\nonumber\\
&&\times
\sum_{l_3=\frac{m+1}{2}}^{m-1}e_m\left(-a_3l_3\right)
\sum_{l_4=\frac{m+1}{2}}^{m-1}e_m\left(-a_4l_4\right)+O\left(\frac{m^4(\log m)^3}{k}\right)\nonumber\\
&=&\sum_{1\leq|a_1|, |a_2|\leq\frac{5k}{32}}~~\mathop{\mathop{
\sum_{1\leq|a_3|, |a_4|\leq\frac{k}{32}}
}_{a_1+a_2+a_3+a_4=0}}_{a_2+2a_3+3a_4=0}~~
\sum_{l_1=\frac{m+1}{2}}^{m-1}e_m\left(-a_1l_1\right)
\sum_{l_2=\frac{m+1}{2}}^{m-1}e_m\left(-a_2l_2\right)\nonumber\\
&&\times
\sum_{l_3=\frac{m+1}{2}}^{m-1}e_m\left(-a_3l_3\right)
\sum_{l_4=\frac{m+1}{2}}^{m-1}e_m\left(-a_4l_4\right)+O\left(\frac{m^4(\log m)^3}{k}\right)\nonumber\\
&=&\sum_{1\leq|a_3|, |a_4|\leq\frac{k}{32}}~~
\sum_{l_1=\frac{m+1}{2}}^{m-1}e_m\left(-(a_3+2a_4)l_1\right)
\sum_{l_2=\frac{m+1}{2}}^{m-1}e_m\left((2a_3+3a_4)l_2\right)\nonumber\\
&&\times\sum_{l_3=\frac{m+1}{2}}^{m-1}e_m\left(-a_3l_3\right)
\sum_{l_4=\frac{m+1}{2}}^{m-1}e_m\left(-a_4l_4\right)+O\left(\frac{m^4(\log m)^3}{k}\right).
\end{eqnarray*}
It is not hard to show that
\begin{eqnarray*}
&&\sum_{\frac{k}{32}<|a_3|\leq\frac{m-1}{2}}~~
\sum_{1\leq|a_4|\leq\frac{k}{32}}~~
\sum_{l_1=\frac{m+1}{2}}^{m-1}e_m\left(-(a_3+2a_4)l_1\right)
\sum_{l_2=\frac{m+1}{2}}^{m-1}e_m\left((2a_3+3a_4)l_2\right)\nonumber\\
&&\qquad\times\sum_{l_3=\frac{m+1}{2}}^{m-1}e_m\left(-a_3l_3\right)
\sum_{l_4=\frac{m+1}{2}}^{m-1}e_m\left(-a_4l_4\right) \nonumber\\
&&\ll\sum_{\frac{k}{32}<|a_3|\leq\frac{m-1}{2}}~~
\sum_{1\leq|a_4|\leq\frac{k}{32}}
m\cdot\left|\sum_{l_2=\frac{m+1}{2}}^{m-1}e_m\left((2a_3+3a_4)l_2\right)\right|
\cdot\frac{m}{k}\cdot\frac{m}{|a_4|}\nonumber\\
&&\ll\frac{m^3}{k}\sum_{1\leq|a_4|\leq\frac{k}{32}}\frac{1}{|a_4|}
\sum_{\frac{k}{32}<|a_3|\leq\frac{m-1}{2}}
\left|\sum_{l_2=\frac{m+1}{2}}^{m-1}e_m\left((2a_3+3a_4)l_2\right)\right|\nonumber\\
&&\ll\frac{m^4(\log m)^2}{k}.
\end{eqnarray*}
Therefore
\begin{eqnarray}\label{Lemma-exponentialsums-1}
\Xi_{m,k}&=&\sum_{1\leq|a_3|, |a_4|\leq\frac{m-1}{2}}~~
\sum_{l_1=\frac{m+1}{2}}^{m-1}e_m\left(-(a_3+2a_4)l_1\right)
\sum_{l_2=\frac{m+1}{2}}^{m-1}e_m\left((2a_3+3a_4)l_2\right)\nonumber\\
&&\quad\times\sum_{l_3=\frac{m+1}{2}}^{m-1}e_m\left(-a_3l_3\right)
\sum_{l_4=\frac{m+1}{2}}^{m-1}e_m\left(-a_4l_4\right)+O\left(\frac{m^4(\log m)^3}{k}\right)\nonumber\\
&=&\sum_{\frac{m+1}{2}\leq l_1, l_2, l_3, l_4\leq m-1}~~
\sum_{1\leq|a_3|\leq\frac{m-1}{2}}e_m\left((-l_1+2l_2-l_3)a_3\right)  \nonumber\\
&&\quad\times\sum_{1\leq|a_4|\leq\frac{m-1}{2}}e_m\left((-2l_1+3l_2-l_4)a_4\right)
+O\left(\frac{m^4(\log m)^3}{k}\right) \nonumber\\
&=&\sum_{\frac{m+1}{2}\leq l_1, l_2, l_3, l_4\leq m-1}~~
\sum_{|a_3|\leq\frac{m-1}{2}}e_m\left((-l_1+2l_2-l_3)a_3\right)  \nonumber\\
&&\quad\times\sum_{|a_4|\leq\frac{m-1}{2}}e_m\left((-2l_1+3l_2-l_4)a_4\right) \nonumber\\
&&-\sum_{\frac{m+1}{2}\leq l_1, l_2, l_3, l_4\leq m-1}~~
\sum_{|a_3|\leq\frac{m-1}{2}}e_m\left((-l_1+2l_2-l_3)a_3\right)  \nonumber\\
&&-\sum_{\frac{m+1}{2}\leq l_1, l_2, l_3, l_4\leq m-1}~~
\sum_{|a_4|\leq\frac{m-1}{2}}e_m\left((-2l_1+3l_2-l_4)a_4\right) \nonumber\\
&&+\sum_{\frac{m+1}{2}\leq l_1, l_2, l_3, l_4\leq m-1}1
+O\left(\frac{m^4(\log m)^3}{k}\right) \nonumber\\
&=&m^2\mathop{\mathop{
\sum_{\frac{m+1}{2}\leq l_1, l_2, l_3, l_4\leq m-1}
}_{2l_2\equiv l_1+l_3 (\bmod m)}}_{3l_2\equiv 2l_1+l_4 (\bmod m)}1
-\frac{m(m-1)}{2}\mathop{
\sum_{\frac{m+1}{2}\leq l_1, l_2, l_3\leq m-1}
}_{2l_2\equiv l_1+l_3 (\bmod m)}1 \nonumber\\
&&-\frac{m(m-1)}{2}\mathop{
\sum_{\frac{m+1}{2}\leq l_1, l_2, l_4\leq m-1}
}_{3l_2\equiv 2l_1+l_4 (\bmod m)}1+\frac{(m-1)^4}{16}\nonumber\\
&&+O\left(\frac{m^4(\log m)^3}{k}\right).
\end{eqnarray}

By using the idea of Lemma 2.2 of \cite{LiuL2022}
we have
\begin{eqnarray}\label{Lemma-exponentialsums-2}
\mathop{\sum_{\frac{m+1}{2}\leq l_1, l_2, l_3\leq m-1}}
_{2l_2\equiv l_1+l_3 (\bmod m)}1=\frac{m^2}{8}+O(m),
\end{eqnarray}
\begin{eqnarray}\label{Lemma-exponentialsums-3}
\mathop{\sum_{\frac{m+1}{2}\leq l_1, l_2, l_4\leq m-1}}_{3l_2\equiv 2l_1+l_4 (\bmod m)}1
=\frac{m^2}{8}+O(m)
\end{eqnarray}
and
\begin{eqnarray}\label{Lemma-exponentialsums-4}
\mathop{\mathop{\sum_{\frac{m+1}{2}\leq l_1, l_2, l_3, l_4\leq m-1}}_{2l_2\equiv l_1+l_3 (\bmod m)}
}_{3l_2\equiv 2l_1+l_4 (\bmod m)}1
=\frac{m^2}{12}+O(m).
\end{eqnarray}
Combining (\ref{Lemma-exponentialsums-1})-(\ref{Lemma-exponentialsums-4}) we immediately get
\begin{eqnarray*}
\Xi_{m,k}&=&m^2\left(\frac{m^2}{12}+O(m)\right)
-2\cdot\frac{m(m-1)}{2}\left(\frac{m^2}{8}+O(m)\right)+\frac{(m-1)^4}{16} \\
&&+O\left(\frac{m^4(\log m)^3}{k}\right)\\
&=&\frac{1}{48}m^4+O\left(\frac{m^4(\log m)^3}{k}\right).
\end{eqnarray*}
This completes the proof of Lemma \ref{Lemma-exponentialsums}.
\end{proof}

\bigskip

\section{Correlation measures of order $4$}

\quad \ Now we prove Theorem \ref{Theorem-Sequence-Eulerquotients-generalized}.
By the orthogonality relations of additive character sums we get
$$
s_t=\frac{1}{m}\sum_{|a|\leq\frac{m-1}{2}}\sum_{l=\frac{m+1}{2}}^{m-1}e_m\left(a(Q_{m}(t)-l)\right).
$$
Hence,
$$
(-1)^{s_t}=1-2s_t=-\frac{2}{m}\sum_{1\leq|a|\leq\frac{m-1}{2}}
\sum_{l=\frac{m+1}{2}}^{m-1}e_m\left(-al\right)e_m\left(aQ_{m}(t)\right)+\frac{1}{m}.
$$

Define
$$
\chi_{km}(n)=\left\{
\begin{array}{ll}
e_m\left(Q_m(n)\right), & \hbox{if \ } \gcd(n,m)=1, \\
0, & \hbox{if \ } \gcd(n,m)>1.
\end{array}
\right.
$$
Clearly $\chi_{km}(n+km)=\chi_{km}(n)$, and by (\ref{Eulerquotients1}) we have
$\chi_{km}(n_1n_2)=\chi_{km}(n_1)\chi_{km}(n_2)$. Then $\chi_{km}(n)$ is a Dirichlet
character modulo $km$ with $\chi_{km}^m$ being trivial. Therefore
\begin{eqnarray}\label{Sequence-Eulerquotients-express}
(-1)^{s_t}=-\frac{2}{m}\sum_{1\leq|a|\leq\frac{m-1}{2}}
\sum_{l=\frac{m+1}{2}}^{m-1}e_m\left(-al\right)\chi_{km}\left(t^a\right)
+\frac{1}{m}.
\end{eqnarray}

By (\ref{Sequence-Eulerquotients-express}), Lemmas \ref{Lemma-charactersums} and
\ref{Lemma-exponentialsums} we get
\begin{eqnarray}\label{Correlation-Eulerquotients1-0}
&&\sum_{t=0}^{km-1}(-1)^{s_t+s_{t+m}+s_{t+2m}+s_{t+3m}}\nonumber\\
&&=\mathop{\sum_{t=0}^{km-1}}_{\gcd(t,m)=1}(-1)^{s_t+s_{t+m}+s_{t+2m}+s_{t+3m}}
+\mathop{\sum_{t=0}^{km-1}}_{\gcd(t,m)>1}1 \nonumber\\
&&=\frac{2^4}{m^4}\sum_{1\leq|a_1|\leq\frac{m-1}{2}}~~
\sum_{l_1=\frac{m+1}{2}}^{m-1}e_m\left(-a_1l_1\right)
\sum_{1\leq|a_2|\leq\frac{m-1}{2}}~~
\sum_{l_2=\frac{m+1}{2}}^{m-1}e_m\left(-a_2l_2\right) \nonumber\\
&&\qquad\times\sum_{1\leq|a_3|\leq\frac{m-1}{2}}~~
\sum_{l_3=\frac{m+1}{2}}^{m-1}e_m\left(-a_3l_3\right)\sum_{1\leq|a_4|\leq\frac{m-1}{2}}~~
\sum_{l_4=\frac{m+1}{2}}^{m-1}e_m\left(-a_4l_4\right) \nonumber\\
&&\qquad\times
\sum_{t=0}^{km-1}\chi_{km}\left(t^{a_1}(t+m)^{a_2}(t+2m)^{a_3}(t+3m)^{a_4}\right) \nonumber\\
&&\qquad+\mathop{\sum_{t=0}^{km-1}}_{\gcd(t,m)>1}1 +O\left(k(\log m)^3\right) \nonumber\\
&&=\frac{2^4k\phi(m)}{m^4}\mathop{\mathop{
\sum_{1\leq|a_1|, |a_2|, |a_3|, |a_4|\leq\frac{m-1}{2}}
}_{a_1+a_2+a_3+a_4\equiv 0\ (\bmod m)}}_{a_2+2a_3+3a_4\equiv 0\ (\bmod k)}~~
\sum_{l_1=\frac{m+1}{2}}^{m-1}e_m\left(-a_1l_1\right)
\sum_{l_2=\frac{m+1}{2}}^{m-1}e_m\left(-a_2l_2\right)\nonumber\\
&&\qquad\times
\sum_{l_3=\frac{m+1}{2}}^{m-1}e_m\left(-a_3l_3\right)
\sum_{l_4=\frac{m+1}{2}}^{m-1}e_m\left(-a_4l_4\right) \nonumber\\
&&\qquad+k(m-\phi(m))+O\left(\phi(m)\phi(k)^{-1}k^{\frac{3}{2}}(\log m)^4\right)+O\left(k(\log m)^3\right) \nonumber\\
&&=km-\frac{2}{3}k\phi(m)+O\left(\phi(m)\phi(k)^{-1}k^{\frac{3}{2}}(\log m)^4\right).
\end{eqnarray}
This proves Theorem \ref{Theorem-Sequence-Eulerquotients-generalized}.

\section{Final remarks}

In this work, we have claimed that two families of binary sequences (see
(\ref{Sequence-Eulerquotients-1}) and (\ref{Sequence-Eulerquotients-2})) studied in the past several years
have `large' values on the correlation measures of order $4$. They would be very
vulnerable if used in cryptography.

It seems interesting to consider the case when the full peaks on the periodic correlation
measure of these sequences appear, i.e., their periodic correlation
measure of order $k$ equals to the period, see \cite{ChenGGT2021}. Such problem may be related to their linear complexity.\\

\textbf{Acknowledgements.} H. Liu was partially supported by National Natural Science Foundation of China under
Grant No. 12071368, and the Science and Technology Program of
Shaanxi Province of China under Grant No. 2019JM-573 and 2020JM-026.

Z. Chen and C. Wu were partially supported by the National Natural Science
Foundation of China under grant No. 61772292,
by the Provincial Natural Science Foundation of Fujian under grant No. 2020J01905 and
by the Program for Innovative Research Team in Science and Technology in Fujian Province University under grant No. 2018-49.


\end{document}